\documentclass[a4paper, 11pt]{amsart}
\usepackage{amssymb, amsthm,amsmath}
\usepackage{enumerate}
\usepackage{stmaryrd}
\usepackage{mathrsfs}
\PassOptionsToPackage{hyphens}{url}
\usepackage[all]{xy}
\xyoption{rotate}
\usepackage{soul}
\usepackage{hyperref}

\numberwithin{equation}{section}
\newtheorem{thm}[equation]{Theorem}
\newtheorem{lem}[equation]{Lemma}
\newtheorem{prop}[equation]{Proposition}
\newtheorem{cor}[equation]{Corollary}

\newtheorem{question}[equation]{Open Question}
\newtheorem{introthm}{Theorem}

\theoremstyle{definition}
\newtheorem{defn}[equation]{Definition}
\newtheorem{define}[equation]{Definition}
\theoremstyle{remark}
\newtheorem{rem}[equation]{Remark}

\newtheorem{example}[equation]{Example}

\newcommand{\Hom}{\operatorname{Hom}}
\newcommand{\End}{\operatorname{End}}
\newcommand{\Ext}{\operatorname{Ext}}
\newcommand{\Ind}{\operatorname{Ind}}
\newcommand{\Res}{\operatorname{Res}}
\newcommand{\FP}{\operatorname{FP}}
\newcommand{\PFP}{\operatorname{PFP}}

\usepackage{color}

\newcommand{\comm}[1]{}

\begin{document}

\title[Probabilistic Finiteness Properties for Profinite Groups]{Probabilistic Finiteness Properties \\ for Profinite Groups}
\author{Ged Corob Cook, Matteo Vannacci}
\date{}
\thanks{The first author was supported by ERC grant 336983 and Basque government grant IT974-16. The second author acknowledges support from the research training group \textit{GRK 2240: Algebro-Geometric Methods in Algebra, Arithmetic and Topology}, funded by the DFG}

\begin{abstract}
We introduce various probablistic finiteness conditions for profinite groups related to positive finite generation (PFG). We investigate completed group rings which are PFG as modules, and use this to answer \cite[Question 1.2]{KV} on positively finitely related groups. Using the theory of projective covers, we define and characterise a probabilistic version of the $\FP_n$ property for profinite groups, called $\PFP_n$. Finally, we prove how these conditions are related to previously defined finiteness conditions and each other.
\end{abstract}

\maketitle

\section*{Introduction}
\subsection*{Motivation}

The study of finiteness properties of abstract groups has a long history; see \cite{Brown} for some background. Analogously to abstract groups, a profinite group $G$ is said to be of \emph{type $\FP_n$} over a profinite ring $R$ if there is a projective resolution $\ldots \to P_n \to \ldots \to P_1\to P_0 \to R \to 0$ with $P_0, \ldots, P_n$ finitely generated profinite $R \llbracket G \rrbracket$-modules. 

Even the first steps into studying this property run into some difficulties for profinite groups that do not occur in the abstract case; for instance $\FP_1$ and finite generation are not equivalent (see \cite{Damian}). 

In this paper we consider an alternative notion of finite generation. The class of \emph{positively finitely generated} groups (PFG groups for short) was introduced in \cite{Mann} and it consists of those profinite groups $G$ where, for some $k$, $k$ Haar-random elements generate $G$ with positive probability (cf.\ Section \ref{sec:pfgpfrmore}). \emph{Positive finite relatedness} (PFR) was introducted in \cite{KV} as a higher analogue of PFG (see Section \ref{sec:pfgpfrmore} for the definition).

In the same spirit we will define and study some related `probabilistic' finiteness properties, such as \emph{positively finitely presented} profinite groups, and 
%
we introduce a new family of higher finiteness properties for profinite groups via the concept of \emph{PFG modules} and \emph{modules of type $\PFP_n$}. 

\subsection*{Main results}

For convenience, we provide a diagram showing the relationships between the various conditions studied in the paper.

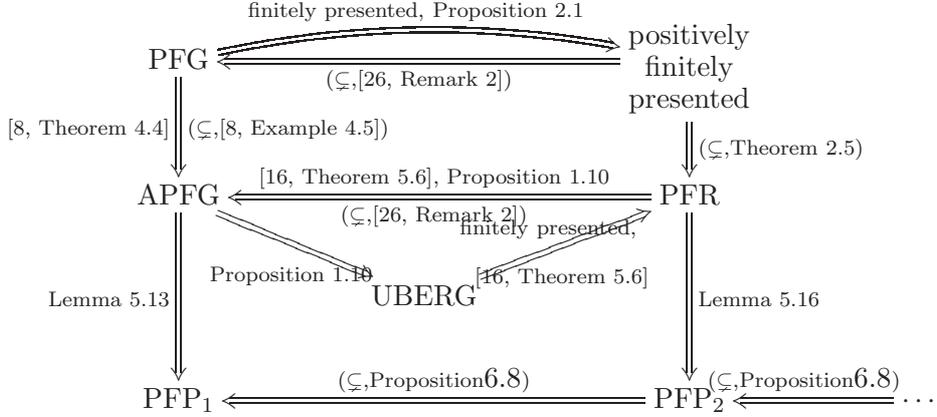
\[\xymatrix@C=50pt{
\text{PFG} \ar@/^1.0pc/@{=>}[rr]^-{\text{finitely presented, Proposition \ref{prop:PFGimpliesPFR}}} \ar@{=>}[d]_{\text{\cite[Theorem 4.4]{Damian}}}^{(\subsetneq, \text{\cite[Example 4.5]{Damian}})}
		&& \txt{positively\\finitely\\presented} \ar@{=>}[ll]^-{(\subsetneq,\text{\cite[Remark 2]{Vannacci}})} \ar@{=>}[d]^-{(\subsetneq, \text{Theorem \ref{PFRnotPFG}})} \\
\text{APFG} \ar@{=>}[dr]_-*!/^12pt/[@]{\labelstyle \text{Proposition \ref{APFGring}}} \ar@{=>}[dd]_{\text{Lemma \ref{APFGPFP1}}}
		&& \text{PFR} \ar@{=>}[ll]_{\text{\cite[Theorem 5.6]{KV}, Proposition \ref{APFGring}}}^{(\subsetneq,\text{\cite[Remark 2]{Vannacci}})} \ar@{=>}[dd]^{\text{Lemma \ref{pfp2}}} \\
	& \text{UBERG} \ar@{=>}[ur]^(0.45)*!/_16pt/[@]{\labelstyle \text{finitely presented,}}_*!/^12pt/[@]{\labelstyle \text{\cite[Theorem 5.6]{KV}}} \\
\PFP_1
		&& \PFP_2 \ar@{=>}[ll]_{(\subsetneq, \text{Proposition} \ref{FPnnotn+1})}
			& \cdots \ar@{=>}[l]_{(\subsetneq, \text{Proposition} \ref{FPnnotn+1})}
}\]

Implications in this diagram which follow trivially from the definitions have been left without references; for those marked $(\subsetneq,-)$, this reference provides a counterexample showing the reverse implication fails; for those marked `finitely presented', the implication holds for finitely presented profinite groups, though not in general.

In Section \ref{sec:APFGdefn} we consider two natural modules associated to a profinite group $G$: the group ring $\hat{\mathbb{Z}}\llbracket G \rrbracket$ and the augmentation ideal $I_{\hat{\mathbb{Z}}}\llbracket G \rrbracket$. When $\hat{\mathbb{Z}}\llbracket G \rrbracket$ is PFG (as a $\hat{\mathbb{Z}}\llbracket G \rrbracket$-module; see Section \ref{sec:APFGdefn}) we say that $G$ has \emph{UBERG}. If the ideal $I_{\hat{\mathbb{Z}}}\llbracket G \rrbracket$ is PFG as a $\hat{\mathbb{Z}}\llbracket G \rrbracket$-module we say that $G$ is \emph{APFG}. We show that APFG implies UBERG, and the converse is true if $G$ has type $\FP_1$ over $\hat{\mathbb{Z}}$ (see Proposition \ref{APFGring}).

In Section \ref{sec:PFGPFR} we answer \cite[Question 1.2]{KV}, by showing:

\begin{introthm}
If $G$ is finitely presented and PFG, it is PFR, but not all PFR groups are PFG.
\end{introthm}

In Section \ref{sec:posfinpres} we introduce the notion of \emph{positively} \emph{fi\-ni\-te\-ly} \emph{presented} profinite groups, as a natural analogue of the definition of finitely presented groups -- that is, groups that have a `PFG presentation' (see Definition \ref{defn:posfinpres}) -- based on the fact that all PFG groups admit an epimorphism from a PFG projective group. We prove some important properties of positively finitely presented groups (see Section \ref{sec:frattinicovers} for the definition of universal Frattini cover of a profinite group).

\begin{introthm}[Properties of positively finitely presented groups]\label{thm:posfinpres}
\quad
\begin{enumerate}[(i)]
\item The class of positively finitely presented profinite groups is closed under extensions (see Lemma \ref{pfp-extension}).
\item  A profinite group $G$ is positively finitely presented if and only if the kernel of the universal Frattini cover $\tilde{G} \to G$ of $G$ is positively normally finitely generated in $\tilde{G}$ (see Proposition \ref{pfp-frattinicover}).
\item  In the class of PFG groups, positive finite presentability is equivalent to PFR (see Corollary \ref{pfp=pfr}).
\end{enumerate}
\end{introthm}


In Section \ref{sec:PFP_modules} we start our investigation of positively finite generated (PFG) \emph{modules} and modules of \emph{type $\PFP_n$}
(see Sections \ref{sec:APFGdefn} and \ref{sec:PFPn_modules}). As for groups, any PFG module admits an epimorphism from a PFG projective module, so we can form projective resolutions from these, and define a module to have type $\PFP_n$ if it has a projective resolution $P_\ast$ in which $P_0,\ldots,P_n$ are PFG. Additionally, we develop several fundamental tools to work with PFG modules and we obtain a characterisation (in the spirit of the Mann-Shalev Theorem) of modules of type $\PFP_n$ in terms of growth conditions on the sizes of certain $\mathrm{Ext}$-groups (see Theorem \ref{extcondition}). 

As usual, there is a similar definition of type $\PFP_n$ for groups. The first crucial observation, using the theory of projective covers, is that a profinite group has type $\PFP_0$ over a commutative ring $R$ if and only if $R$ is PFG; and it turns out that type $\PFP_n$ coincides with type $\FP_n$ in the class of prosoluble groups (see Remark \ref{ex:PFPnprosoluble}).

Our next main result is an equivalent cohomological characterisation of type $\PFP_n$ obtained by applying Theorem \ref{extcondition}. For a commutative profinite ring $R$ and $k\in \mathbb{N}$, denote by $\mathcal{S}^{R \llbracket G \rrbracket}_k$ the set of irreducible $R \llbracket G \rrbracket$-modules of order $k$.

\begin{introthm}
\label{extconditionintro}
Let $G$ be a profinite group and $R$ a commutative profinite ring. Then, $G$ has type $\PFP_n$ over $R$ if and only if $\sum_{S \in \mathcal{S}^{R \llbracket G \rrbracket}_k} (\lvert H_R^m(G,S) \rvert-1)$ has polynomial growth in $k$ for all $m \leq n$.
\end{introthm}

When $G$ is known to have type $\FP_n$ over $R$, we get another characterisation of type $\PFP_n$, from Corollary \ref{Damian3}.

\begin{introthm}
Let $G$ be a profinite group of type $\FP_n$ and $R$ a commutative profinite ring. Then, $G$ has type $\PFP_n$ over $R$ if and only if $$\lvert \{ S \in \mathcal{S}^{R \llbracket G \rrbracket}_k: H_R^m(G,S) \neq 0\} \rvert$$ has polynomial growth in $k$ for all $m \leq n$.
\end{introthm}

Next, we show that the class of groups of type $\PFP_n$ is closed under commensurability and finite direct products. The proof of the next theorem can be found in Proposition \ref{prop:PFPncommensurability} and Proposition \ref{prop:PFPndirectprod}.

\begin{introthm}[Closure properties of groups of type $\PFP_n$]
\begin{enumerate}[(i)]
\item Let $G$ be a profinite group and let $H$ an open subgroup of $G$. Then $G$ is of type $\PFP_n$ if and only if $H$ is.
\item Let $G_1$ and $G_2$ be profinite groups of type $\PFP_n$ and $\PFP_m$ over $R$, respectively. Then $G_1\times G_2$ is of type $\PFP_{\min(n,m)}$ over $R$.
\end{enumerate}
\end{introthm}

Up to this point, property $\PFP_n$ seems to enjoy several nice properties, but it still remains mysterious. To address this shortcoming, in Section \ref{sec:APFG} and \ref{PFPvsPFP2} we start examining properties $\PFP_1$ and $\PFP_2$ in more detail. 

In Section \ref{sec:APFG} we show that APFG implies $\PFP_1$ (Lemma \ref{APFGPFP1}).  Using this, we can give an example of a group of type $\PFP_1$ that is not PFG (see Example \ref{ex:damian}). 


In Section \ref{PFPvsPFP2}, we investigate the relation between PFR and type $\PFP_2$. We show in Lemma \ref{pfp2} that PFR implies $\PFP_2$, and in Proposition \ref{PFP2=absubgps} that type $\PFP_2$ can be detected by considering the minimal presentation of a group.

In the case of modules for abstract groups, or the usual (non-probabilistic) definition of finite generation for profinite modules, we have the following nice property: if $M$ is an $H$-module with $H \leq G$, $M$ is finitely generated if and only if $\Ind^G_H M$ is. But it is not hard to show that an analogous property fails for positive finite generation. For example, $\hat{\mathbb{Z}}$ is PFG, but by Proposition~\ref{prop:free_not_PFP}, $\Ind_1^{F_n} \hat{\mathbb{Z}}$ is not PFG, for $F_n$ the free profinite group on $n>1$ generators. We confront this problem in Section \ref{sec:relpfpn}. Specifically, we define a relative version of type $\PFP_n$. Given a profinite group $G$ and a closed normal subgroup $H$ of $G$, we say that $H$ has \emph{relative type $\PFP_n$ in $G$} if all PFG projective $R \llbracket G/H \rrbracket$-modules have type $\PFP_n$ over $R \llbracket G \rrbracket$. Note that in the analogous type $\FP_n$ case, relative type $\FP_n$ is equivalent to type $\FP_n$.

\begin{introthm} Suppose that $H$ has relative type $\PFP_m$ in $G$ over $R$.
\begin{enumerate}[(i)]
\item If $G$ has type $\PFP_n$ over $R$, then $G/H$ has type $\PFP_{\min(m+1,n)}$ over $R$.
\item If $G/H$ has type $\PFP_n$ over $R$, then $G$ has type $\PFP_{\min(m,n)}$ over $R$.
\end{enumerate}
\end{introthm}

This is shown in Corollary \ref{relPFPn2}. Incidentally, the proof of the previous theorem also works in showing the corresponding result for abstract groups and our proof does not depend on the usual spectral sequence argument, which may have some independent interest.

We finish with Section \ref{examples}, where we produce some novel examples to distinguish some of the aforementioned classes of groups. First, for any prime $p$, we give examples of type $\PFP_1$ groups $G$ over $\mathbb{Z}_p$ such that the group ring $\mathbb{Z}_p\llbracket G \rrbracket$ is not PFG (see Proposition \ref{prop:finsimplegroup}). We would like to give examples of this behaviour over $\hat{\mathbb{Z}}$, to distinguish between the classes of groups with UBERG and type $\PFP_1$ over $\hat{\mathbb{Z}}$, but for now the question remains open. Such examples cannot appear among pronilpotent groups (see Proposition \ref{prop:pronilp} and the related Question \ref{quest:prosoluble} about the prosoluble case).

Although, trivially, if a module has type $\PFP_n$ it has type $\FP_n$, we show that the converse is not true, by showing in Section \ref{FP1vsPFP1} that the free profinite group on $m$ generators, $1 < m < \infty$, has type $\FP_1$ but not type $\PFP_1$ over $\hat{\mathbb{Z}}$.

Finally, in Proposition \ref{FPnnotn+1}, we distinguish the classes of groups of type $\PFP_n$ over $\hat{\mathbb{Z}}$ by constructing, for each $n \geq 0$, a prosoluble group of type $\PFP_n$ but not $\PFP_{n+1}$.

\section{Preliminaries and notation}

We state now some conventions which will be in force for the rest of this article. All subgroups and submodules will be assumed to be closed. Generation will always be intended in the topological sense. All homomorphisms will be continuous. Modules will be assumed to be left modules.


\subsection{Homology theory for profinite groups}


In the course of this work, we will need the usual `homological lemmas' for profinite groups, such as snake lemma, horseshoe lemma, Schanuel's lemma, Lyndon-Hochschild-Serre spectral sequence, the long exact sequence in cohomology, and Ext groups -- see for instance \cite{RZ}. Other tools such as the mapping cone construction, valid in all abelian categories, can be found in \cite{Weibel} for example.

\subsection{Haar measure}

For a profinite group $G$, we denote by $\mu_G$ the (left) \emph{Haar measure} of $G$, see \cite[Chapter 11]{LS} for basic properties. We will always consider the \emph{normalised} Haar measure, in this way we can turn a profinite group into a probability space.

For a profinite group $G$, the direct power $G^k$ can be viewed as a profinite group, and as such it supports a Haar measure. For $g_1,\ldots,g_k \in G$, we denote by $\langle g_1,\ldots,g_k \rangle$ the closed subgroup of $G$ they generate.

Now we define the set
\[
  X(G,k) = \{(g_1,\ldots,g_k)\in G^k \mid \langle g_1\ldots,g_k\rangle = G\},
\] 
so that $P(G,k) = \mu_{G^k}(X(G,k))$ is non-zero for some $k$ if and only if $G$ is PFG.

\begin{lem}[{\cite[Lemma 11.1.1]{LS}}]\label{lem:11.1.1}
Let $G$ be a profinite group, let $K$ be a closed normal subgroup of $G$ and $\pi:G\to G/K$ be the natural projection. If $X$ is a closed subset of $G$, then $\mu_{G/K}(\pi(X))\ge \mu_G(X)$; if $Y$ is a closed subset of $G/K$, then $\mu_G(\pi^{-1}(Y)) = \mu_{G/K}(Y)$. 
\end{lem}

\subsection{PFG, PFR and more}\label{sec:pfgpfrmore}

We say that a profinite group $G$ is \emph{PFG} if there is a positive integer $k$ such that the probability of $k$ Haar-random elements of $G$ generating the whole group is positive. This condition has been studied extensively and here we only mention the Mann-Shalev theorem \cite[Theorem 4]{MS}: a profinite group $G$ is PFG if and only if it has polynomial maximal subgroup growth.

In the spirit of the Mann-Shalev theorem, the authors of \cite{KV} study a related property called PFR. We list below some of the conditions considered there that we will need; the interested reader may check \cite{KV} for more details.

A profinite group $G$:
\begin{enumerate}[(i)]
\item is \emph{PFR} if it is finitely generated, and for every epimorphism $f:H \to G$ with $H$ finitely generated, the kernel of $f$ is positively finitely normally generated in $H$;
\item has PMEG, if the number of isomorphism classes of minimal extensions of $G$ of order $n$ grows polynomially in $n$;
\item has UBERG, if the number of irreducible $\hat{\mathbb{Z}}\llbracket G \rrbracket$-modules of order $n$ grows polynomially in $n$.
\end{enumerate}

\begin{prop}[\cite{KV}]
\label{KV}
UBERG is equivalent to the group algebra $\hat{\mathbb{Z}}\llbracket G \rrbracket$ being PFG, for all profinite groups $G$. PFR and PMEG are equivalent for $G$ finitely generated. All the conditions are equivalent for $G$ finitely presented.
\end{prop}

Note that the equivalence of UBERG to $\hat{\mathbb{Z}}\llbracket G \rrbracket$ being PFG is only stated in \cite{KV} for finitely generated groups, but the proof for general groups goes through without change.


We will see later that there are groups with UBERG which are not PFG; the question of whether there are non-finitely generated groups with UBERG remains open. But we do have the following result.

\begin{prop}
\label{UBERG=cb}
Suppose $G$ has UBERG, then it is countably based.
\end{prop}
\begin{proof}
Suppose by contradiction that $G$ is not countably based. Then, for some index $n$, there are infinitely many open normal subgroups of index $n$ in $G$ and so $G$ has infinitely many irreducible representations of dimension at most $n$ (over fields of characteristic $p > n$).
%
%
\end{proof}

\subsection{PFG modules}

For a module $M$, as for a profinite group, the direct power $M^k$ can be viewed as an abelian profinite group, and supports a Haar measure which we will denote by $\mu_{M^k}$. For $m_1,\ldots,m_k \in M$, we denote by $\langle m_1,\ldots,m_k \rangle_R$ the closed submodule of $M$ they generate.

As for groups, for a positive integer $k$, we define the set
\[
  X_R(M,k) = \{(m_1,\ldots,m_k)\in M^k \mid \langle m_1\ldots,m_k\rangle_R = M\}.
\] 
and $P_R(M,k) = \mu_{M^k}(X_R(M,k))$.

\begin{defn}
A profinite module $M$ is said to be \emph{PFG} if there is some $k\in \mathbb{N}$ such that $P_R(M,k) > 0$. 
\end{defn}

In \cite[Proposition 11.2.1]{LS} it is shown that the class of PFG groups is closed under quotients and extensions. The same is true for PFG modules with an analogous proof which we omit.

\begin{lem}\label{lem:PFGquotext}
PFG modules are closed under quotients and extensions.
\end{lem}

\subsection{APFG}
\label{sec:APFGdefn}

It will be useful here to compare our conditions on profinite groups to another condition introduced in \cite{Damian}. Recall that the \emph{augmentation map} $\varepsilon: \hat{\mathbb{Z}}\llbracket G \rrbracket \to \hat{\mathbb{Z}} $ is induced by $\varepsilon(g)=1$, for $g\in G$. Define the \emph{augmentation ideal} as $I_{\hat{\mathbb{Z}}}\llbracket G \rrbracket = \ker \varepsilon$, so we have a short exact sequence 
\begin{equation}\label{eq:augmentation}
0 \to I_{\hat{\mathbb{Z}}}\llbracket G \rrbracket \to \hat{\mathbb{Z}} \llbracket G \rrbracket \to \hat{\mathbb{Z}} \to 0.
\end{equation}

Note that the group $G$ has type $\FP_1$ if and only if its augmentation ideal $I_{\hat{\mathbb{Z}}}\llbracket G \rrbracket$ is finitely generated.
\begin{defn}
A profinite group $G$ is said to be \emph{APFG} if the augmentation ideal $I_{\hat{\mathbb{Z}}}\llbracket G \rrbracket$ is PFG as a $\hat{\mathbb{Z}} \llbracket G \rrbracket$-module. 
\end{defn}

\begin{rem}
There is an error in the statement of \cite[Corollary 2.5]{Damian} (which does not affect the rest of \cite{Damian}): it should say ``$I_{\hat{\mathbb{Z}}}\llbracket G\rrbracket$ is finitely generated as a $\hat{\mathbb{Z}} \llbracket G \rrbracket$-module if and only if there exists $d \in \mathbb{N}$ such that $\delta_G(M)/r_G(M) \leq d$ for any $M \in \mathrm{Irr}(\mathbb{F}_p\llbracket G \rrbracket)$ and $p \in \pi(G)$''.
\end{rem}

Recall that $\hat{\mathbb{Z}} \llbracket G \rrbracket$ is PFG if and only if $G$ has UBERG. We can now state the following fundamental result.

\begin{prop}[{\cite[Theorem 4.4]{Damian}}]\label{prop:damian_thm4.4}
Every PFG group is APFG.
\end{prop}

It is shown in \cite[Theorem 3]{Damian} that, if $G$ is countably based and has type $\FP_1$ over $\hat{\mathbb{Z}}$, $G$ is APFG if and only if the number of irreducible $\hat{\mathbb{Z}} \llbracket G \rrbracket$-modules has polynomial growth (and, hence, if and only if $G$ has UBERG). Here we may generalise this theorem by dropping the countably based requirement. 

\begin{prop}
\label{APFGring}
Let $G$ be a profinite group. If $G$ is APFG, then $G$ has UBERG. If $G$ has UBERG and type $\FP_1$, then $G$ is APFG.
\end{prop}
\begin{proof}
Suppose that the augmentation ideal $I_{\hat{\mathbb{Z}}}\llbracket G\rrbracket$ of $G$ is PFG. Then the group ring $\hat{\mathbb{Z}} \llbracket G \rrbracket$ fits into the exact sequence \eqref{eq:augmentation} and it is PFG as an extension of two PFG modules.

On the other hand, if $\hat{\mathbb{Z}} \llbracket G \rrbracket$ is PFG, then $\hat{\mathbb{Z}} \llbracket G \rrbracket$-modules are PFG if and only if they are finitely generated; for $G$ of type $\FP_1$, the augmentation ideal is finitely generated and hence PFG.
\end{proof}

Note that profinite groups of type $\FP_1$ over $\hat{\mathbb{Z}}$ which are not countably based exist, by \cite[Example 7.1]{CC}, but they do not have UBERG by Proposition \ref{UBERG=cb}.

Combining the last two propositions, we get:

\begin{cor}
\label{PFGtoUBERG}
If $G$ is PFG, it has UBERG.
\end{cor}

\begin{rem}\label{rem:prop6.1notfingen}
All finitely generated prosoluble groups by \cite[Corollary 6.12]{KV} are PFG, and hence have UBERG. Non-abelian free profinite groups do not have UBERG by \cite[Proposition 6.14]{KV}.
\end{rem}

\subsection{The Schur multiplier and Schur covers}
\label{Schur}

The Schur multiplier of an abstract group is an important source of group-theoretic information. Basic facts we use about the Schur multiplier in this paper may be found in \cite{BT} and \cite{KarpSM}. The profinite version is less well known, so we introduce it here. Recall that the Schur multiplier, for an abstract group $G$, is $$H_2(G,\mathbb{Z}) \cong \frac{R \cap [F,F]}{[F,R]},$$ for a presentation of $G$ given by the exact sequence $1 \to R \to F \to G \to 1$ with $F$ free. The same argument as in the abstract case (see, for example, \cite[Presentations 6.8.7]{Weibel}) shows that, for $G$ profinite and $1 \to \hat{R} \to \hat{F} \to G \to 1$ a profinite presentation, we have $$H_2(G,\hat{\mathbb{Z}}) \cong \frac{\hat{R} \cap [\hat{F},\hat{F}]}{[\hat{F},\hat{R}]}.$$ Here, by $[\hat{F},\hat{F}]$ and $[\hat{F},\hat{R}]$, we mean the closure of the abstract subgroups generated by these commutators. We may write $M(G)$ for the Schur multiplier of a profinite group $G$.

\begin{lem}
For $G$ finite, $H_2(G,\mathbb{Z})$ (with $G$ considered as an abstract group) is isomorphic to $H_2(G,\hat{\mathbb{Z}})$ (with $G$ considered as a profinite group).
\end{lem}
\begin{proof}
It is well-known that the Schur multiplier for abstract groups is equal to $H^2(G,\mathbb{C}^\times)$. Calculating cohomology using the bar resolution, each term in the cochain complex $\Hom(G^n,\mathbb{C}^\times)$ is equal to $\Hom(G^n,\mathbb{Q}/\mathbb{Z})$ because the image of each element of $G^n$ has finite order; we deduce that the Schur multiplier is isomorphic to $H^2(G,\mathbb{Q}/\mathbb{Z})$. By considering the bar resolution again, this is the same as $H^2(G,\mathbb{Q}/\mathbb{Z})$ when we think of $G$ as a profinite group because every abstract homomorphism in $\Hom(G^n,\mathbb{Q}/\mathbb{Z})$ is continuous, which is isomorphic to $H_2(G,\hat{\mathbb{Z}})$ by Pontryagin duality, \cite[Proposition 6.3.6]{RZ}.
\end{proof}

Thus for finite groups we can talk about `the' Schur multiplier, to mean the group defined (up to isomorphism) by both our definitions.

An important part of the theory of Schur multipliers for abstract groups is the existence of Schur covers. There does not seem to be a good reference for Schur covers of profinite groups, but we will content ourselves here with giving the facts we will need: the proofs go through just as they do for abstract groups.

A Schur cover of a profinite group $G$ is an exact sequence $M(G) \rightarrowtail H \twoheadrightarrow G$ which is a stem extension: an extension in which $M(G)$ is contained in the centre of $H$ and in the derived subgroup $\overline{[H,H]}$. Schur covers are precisely the maximal stem extensions of $G$. The universal coefficient theorem gives a short exact sequence $$0 \to \Ext_{\hat{\mathbb{Z}}}^1(G_{ab},M(G)) \to H^2(G,M(G)) \to \Hom(M(G),M(G)) \to 0$$ (here all the cohomological functors are understood to be derived functors of the abstract group of homomorphisms, defined by taking a projective resolution in the first variable), and the preimage of $\mathrm{id}_{M(G)}$ in $H^2(G,M(G))$ gives all the isomorphism classes of Schur covers of $G$ (see \cite[Proposition II.3.2]{BT}). So the number of such isomorphism classes is at most $|\Ext_{\hat{\mathbb{Z}}}^1(G_{ab},M(G))|$. In particular, for $G$ a perfect group, there is a unique Schur cover.

\begin{prop}[{\cite[Theorem II.5.3]{BT}}]
\label{perfectcover}
Let $G$ be perfect, and $H$ its Schur cover. Then $H$ is perfect, and its Schur multiplier $M(H)$ is trivial.
\end{prop}

\subsection{Frattini covers and PFG}
\label{sec:frattinicovers}

For a profinite group $G$, write $\Phi(G)$ for the Frattini subgroup of $G$: that is, the intersection of all the maximal open subgroups of $G$. An epimorphism $f: H \to G$ is called a \emph{Frattini cover} of $G$ if $\ker(f) \leq \Phi(H)$. The Frattini covers of $G$ form an inverse system whose inverse limit, called the \emph{universal Frattini cover} of $G$, is again a Frattini cover of $G$ and is a projective profinite group. See \cite[Chapter 22]{FJ} for background on this.

\begin{lem}
\label{fratcover}
A Frattini cover $H$ of a profinite group $G$ is PFG if and only if $G$ is.
\end{lem}
\begin{proof}
Without loss of generality, we can assume that both $H$ and $G$ are finitely generated. In fact, we just have to observe that, for any generating set $S$ of $G$, any lift of $S$ to $H$ generates $H$, since the kernel is contained in the Frattini subgroup of $H$. The claim follows by Lemma \ref{lem:11.1.1}.
\end{proof}

\begin{lem}
\label{Schurcover}
A Schur cover of a profinite group $G$ is a Frattini cover.
\end{lem}
\begin{proof}
This holds by the same argument as \cite[Lemma 2.4.5]{KarpSM}.
\end{proof}

\section{An open question of Kionke-Vannacci}
\label{sec:PFGPFR}

We first answer the second part of \cite[Open Question 1.2]{KV}.

\begin{prop}\label{prop:PFGimpliesPFR}
Let $G$ be a finitely presented profinite group. If $G$ is PFG, then it is PFR.
\end{prop}
\begin{proof}
By Proposition \ref{prop:damian_thm4.4}, $G$ is APFG. Hence $\hat{\mathbb{Z}} \llbracket G \rrbracket$ fits into \eqref{eq:augmentation} and it is PFG. Now \cite[Theorem A]{KV} shows that $G$ is PFR.
\end{proof}

In the rest of the section, we give an example of a PFR group which is not PFG, answering the first half of \cite[Question 1.2]{KV}.

Consider the group $G = \prod_{n \geq N} A_n^{2^n}$, where $A_n$ is the alternating group on $n$ letters. For the sake of concreteness, we will take $N = 15$; it is entirely possible to take any $N \geq 5$, at the expense of a little extra complexity in the argument, and the features described here will not change, except some constants like the number of generators needed for $G$ and the constant $3$ in the statement of Theorem \ref{PFRnotPFG}. We have tried to make explicit where assumptions on $N$ are actually being used.

\begin{prop}
$G$ is $2$-generated but not PFG.
\end{prop}
\begin{proof}
The argument of \cite[Example 2]{Mann} shows $G$ is not PFG. By \cite[Lemma 5]{KaLu}, it is enough to show $A_n^{2^n}$ is $2$-generated for each $n \geq N$. Let $S$ be a non-abelian finite simple group, and let $D(S)$ be the number of elements of the product $S^2$ which generate $S$. By \cite[Corollary 7]{KaLu}, if $m \leq |D(S)|/|Aut(S)|$, then $S^m$ is $2$-generated. Now for $n \geq 7$, it is well-known that $Aut(A_n) = S_n$, the symmetric group; for $n \geq 5$ we have $|D(A_n)| \geq (1-1/n-8.8/n^2)|A_n|^2$ by \cite[Theorem 1.1]{MRD}. So $|D(S)|/|Aut(S)| \geq (1-1/n-8.8/n^2)n!/4$, which is easily checked to be $>2$ for $n \geq 7$.
\end{proof}

\begin{rem}
On the other hand, for $n=6$ we have $|Aut(A_n)|=2|S_n|$, and the computation using $|D(A_n)| = 0.588|A_n|^2$ from \cite[Table 1]{MRD} shows $A_6^{2^6}$ is not $2$-generated. (The corresponding calculation for $n=5$ shows $A_5^{2^5}$ is not $2$-generated either.)
\end{rem}

For $n \geq 8$, the Schur multiplier of $A_n$ has order $2$ by \cite[Theorem 2.12.5]{KarpSM}, and we write $A^\ast_n$ for the (unique, because $A_n$ is perfect) Schur cover of $A_n$.

By \cite[Theorem 2.2.10]{KarpSM}, the Schur multiplier of a finite product of alternating groups $\prod_i A_{n_i}$ is the product of the Schur multipliers $\prod_i M(A_{n_i}) = \prod_i \mathbb{Z}/2\mathbb{Z}$; with \cite[Theorem 2.8.5]{KarpSM}, it follows that the Schur cover of $\prod_i A_{n_i}$ is $\prod_i A^\ast_{n_i}$. Since profinite homology groups commute with inverse limits, we get that $M(G) = \prod_{n \geq N} (\mathbb{Z}/2\mathbb{Z})^{2^n}$. For the Schur cover, we use the functoriality of our construction: $H^2(G,M(G)) = \prod_j \bigoplus_i H^2(A_{n_i},M(A_{n_j}))$, where $i$ and $j$ index the simple factors of $G$. Now, as described in Section \ref{Schur}, pick the preimage of $\mathrm{id}_{M(G)}$ given by picking $0 \in H^2(A_{n_i},M(A_{n_j}))$ for $i \neq j$ and a preimage of $\mathrm{id}_{M(A_{n_i})}$ in $H^2(A_{n_i},M(A_{n_j}))$ when $i=j$. We deduce that the Schur cover $\tilde{G}$ of $G$ is $\prod_{n \geq N} (A^\ast_n)^{2^n}$ (again, this is unique because $G$ is perfect).

Now $\tilde{G}$ is a Frattini cover of $G$ by Lemma \ref{Schurcover}, so it is $2$-generated and not PFG. We will imitate \cite[Example 4.5]{Damian} to prove:

\begin{prop}
$\tilde{G}$ has UBERG.
\end{prop}
\begin{proof}
We will show the number $g(n)$ of irreducible $\tilde{G}$-modules of order $n$ is polynomial in $n$. Suppose $M$ is an irreducible $\tilde{G}$-module with $|M|=n=p^r$. The action of $\tilde{G}$ on $M$ factors through some finite product $A^\ast_{n_1} \times \cdots \times A^\ast_{n_t}$ such that none of the $A^\ast_{n_i}$ act trivially on it.

Set $u=n_i$ for some $1 \leq i \leq t$ and consider an $A^\ast_u$-composition series for $M$: a filtration $0=M_0 \leq M_1 \leq \cdots \leq M_k=M$ of $A^\ast_u$-submodules such that the factors $K_{l}=M_{l}/M_{l-1}$ are irreducible $A^\ast_u$-modules. Note that $|K_l| > p$ for some $1 \leq l \leq k$, since otherwise we would have a non-trivial homomorphism $A^\ast_u \to U(r,p)$, the group of unitriangular matrices in $GL(r,p)$, which is a $p$-group; this is impossible because the only non-trivial quotients of $A^\ast_u$ are itself and $A_u$.

So $K_l$ is a non-trivial irreducible $A^\ast_u$-module. By \cite[Proposition 5.3.1, Proposition 5.3.7]{KL}, for $u \geq 9$ we get $|K_l| \geq p^{u-2}$. Thus $u-2 \leq \dim M = r$, so $n_i \leq r+2$ for $1 \leq i \leq t$, and every irreducible $\tilde{G}$-module of order $n=p^r$ is an irreducible $H_r$-module, where $H_r = \prod_{j=N}^{r+2}(A^\ast_j)^{2^j}$.

The rest of the proof of \cite[Example 4.5]{Damian} goes through without change.
\end{proof}

Since $M(G)$ is not finitely generated, $G$ does not have type $\FP_2$ over $\hat{\mathbb{Z}}$, so it is not finitely presented.

On the other hand, by \cite[Theorem A]{KV}, to show $\tilde{G}$ is PFR, it suffices to show it is finitely presented. We will do this using the equivalent condition given in \cite[Theorem 0.3]{Lubotzky}: a finitely generated profinite group $H$ is finitely presented if and only if there exists a positive constant $C$ such that, for every prime $p$ and every irreducible $\mathbb{F}_p\llbracket H \rrbracket$-module $M$, $\dim H^2(H,M) \leq C \dim M$.

\begin{thm}
\label{PFRnotPFG}
For $M$ an irreducible $\tilde{G}$-module, $\dim H^2(\tilde{G},M) \leq 3\dim M$. Thus $\tilde{G}$ is finitely presented, and hence PFR.
\end{thm}

We will prove this in several steps.

\begin{lem}
If $M$ is a trivial $\tilde{G}$-module, $H^2(\tilde{G},M)=0$.
\end{lem}
\begin{proof}
By Pontryagin duality, \cite[Proposition 6.3.6]{RZ}, this is equivalent to showing $H_2(\tilde{G},M^\ast)=0$, where $M^\ast = \Hom(M,\mathbb{Q}/\mathbb{Z})$ is the Pontryagin dual of $M$. We have $M \cong \mathbb{F}_p$ for some $p$, so $M^\ast \cong \mathbb{F}_p$ too, also with trivial $\tilde{G}$-action. The short exact sequence $0 \to \hat{\mathbb{Z}} \to \hat{\mathbb{Z}} \to \mathbb{F}_p \to 0$ gives a long exact sequence $$\cdots \to M(\tilde{G})=H_2(\tilde{G},\hat{\mathbb{Z}}) \to H_2(\tilde{G},M^\ast) \to H_1(\tilde{G},\hat{\mathbb{Z}})=\tilde{G}_{ab} \to \cdots,$$ and we have $M(\tilde{G})=\tilde{G}_{ab}=0$ by Proposition \ref{perfectcover}, because $\tilde{G}$ is the Schur cover of a perfect group.
\end{proof}

So for the theorem, we only need to consider $M$ with a non-trivial action. Suppose $M$ is a non-trivial irreducible $\mathbb{F}_p\llbracket \tilde{G} \rrbracket$-module. As before, the action of $\tilde{G}$ on $M$ factors through some finite product $L = A^\ast_{n_1} \times \cdots \times A^\ast_{n_t}$ such that none of the $A^\ast_{n_i}$ act trivially on it, and we write $K$ for the product of all the other quasisimple factors, so $\tilde{G}=K \times L$.

\begin{lem}
$\dim H^2(G,M) \leq \dim H^2(L,M)$.
\end{lem}
\begin{proof}
By the Lyndon-Hochschild-Serre spectral sequence we know
\begin{multline*}
\dim H^2(\tilde{G},M) \leq  \dim H^2(K,M^L) + \\
\dim H^1(K,H^1(L,M)) + \dim H^2(L,M)^K.
\end{multline*}
Since $M$ is irreducible and non-trivial as an $L$-module, $M^L=0$. Since the actions of $K$ on $L$ and $M$ are trivial, the action of $K$ on $H^1(L,M)$ is trivial. Moreover, since $H^1(L,M)$ is abelian and $K$ is perfect, we have $H^1(K,H^1(L,M))=\Hom(K,H^1(L,M))=0$. So $$\dim H^2(G,M) \leq \dim H^2(L,M)^K \leq \dim H^2(L,M).$$
\end{proof}

Write $J$ for the kernel of the $L$-action on $M$: $J$ is a finite product of copies of $\mathbb{Z}/2\mathbb{Z}$.

\begin{lem}
\label{Jcohom}
If $p \neq 2$ or $J$ is trivial, $H^j(J,M)$ is $0$ for $j \geq 1$. If $p=2$ and $J \neq 1$, $H^j(J,M)$ is isomorphic, as an $L/J$-module, to a finite direct sum of copies of $M$.
\end{lem}
\begin{proof}
For $p \neq 2$, this is \cite[Corollary 7.3.3]{RZ}; for $J$ trivial, it is trivial. For $p=2$, we can use the explicit description of the $L/J$-action given in \cite[Section III.8]{Brown}. Using the notation there, we have $\alpha=\mathrm{id}_J$ because $J$ is central in $L$, so we can calculate the $L/J$-action using a finite type free resolution $F_\ast$ of $\mathbb{F}_2$ as a $J$-module, and the action of $L/J$ on $\Hom_J(F_\ast,M)$ is given by $(g \cdot f)(x) = g(f(x))$. Now after fixing a basis for $F_j$, $\Hom_J(F_j,M)$ is clearly a finite direct sum of copies of $M$ indexed by this basis, on which $L/J$ acts diagonally. Since this is a cochain complex of semisimple $L/J$-modules, its homology groups are semisimple, with all their simple factors isomorphic to $M$.
\end{proof}

\begin{lem}
$\dim H^2(L,M) \leq 3\dim M$.
\end{lem}
\begin{proof}
Once again, the Lyndon-Hochschild-Serre spectral sequence gives
\begin{multline*}
\dim H^2(L,M) \leq  \dim H^2(L/J,M) + \\
\dim H^1(L/J,H^1(J,M)) + \dim H^2(J,M)^{L/J}.
\end{multline*}
We know $H^2(J,M)^{L/J}=M^{L/J}=0$ by Lemma \ref{Jcohom}.

If $t>1$, $H^1(L/J,M)=0$ and $\dim H^2(L/J,M) \leq (1/4)\dim M$, by \cite[Lemma 5.2(4)]{GKKL}, and by Lemma \ref{Jcohom} we deduce $H^1(L/J,H^1(J,M))=0$, so $\dim H^2(L,M) \leq (1/4)\dim M$.

If $t=1$ (so $L$ is $A^\ast_n$ for some $n$) and $p \neq 2$, by Lemma \ref{Jcohom} we have $H^1(L/J,H^1(J,M))=0$, and then either $J$ is trivial and $H^2(L,M) = H^2(L/J,M) = 0$ by \cite[Lemma 4.1(2)]{GKKL}, or $J$ is non-trivial, so $J=\mathbb{Z}/2\mathbb{Z}$, $L/J=A_n$, and $$\dim H^2(L,M) =  \dim H^2(L/J,M) \leq \dim M$$ by \cite[Lemma 4.1(2), Theorem 6.2(1)(2)]{GKKL}.

If $t=1$ and $p=2$, by \cite[Lemma 4.1(3)]{GKKL}, $J$ is non-trivial, so $J=\mathbb{Z}/2\mathbb{Z}$ and
\begin{align*}
H^2(L,M) &\leq \dim H^2(L/J,M) + \dim H^1(L/J,M) \\
&\leq (35/12)\dim M + (1/12)\dim M = 3 \dim M
\end{align*}
for $n \geq 15$ by \cite[Theorem 6.1(1), Theorem 6.2(3)]{GKKL}.
\end{proof}

This proves the theorem.

\section{Positively finitely presented groups}\label{sec:posfinpres}


Given that the idea of PFR is an higher analogue of PFG, an alternative condition would require that $G$ has a `PFG presentation'. By Lemma \ref{fratcover}, every PFG group admits a short exact sequence of the form 
\begin{equation}\label{eq:posfinpres}
1 \to R \to P \to G \to 1
\end{equation}
with $P$ a PFG projective profinite group. In this section, we will think of such sequences as presentations for $G$.

Let $G$ be a profinite group and let $A$ be a normal subgroup of $G$. We say that $A$ is \emph{positively finitely normally generated} in $G$ if there exists $k\in \mathbb{N}$ such that, defining the set $$X^{G}(A,k)=\{(a_1,\ldots,a_k)\in A^k \mid \langle a_1,\ldots,a_k \rangle^G = A \},$$ we have $P^{G}(A,k) :=\mu_{A^k}(X^{G}(A,k)) >0$. It is easy to see, by the same argument as \cite[Section 3.3]{KV}, that $A$ is positively finitely normally generated in $G$ if and only if the number of open maximal $G$-stable subgroups of index $n$ in $A$ grows polynomially in $n$. 

\begin{defn}
\label{defn:posfinpres}
A profinite group $G$ is said to be \emph{positively finitely presented}\footnote{We choose to avoid abbreviating this to PFP because of potential clashes with some future paper about a positively type $\FP$ condition.} if $G$ is PFG and for every short exact sequence \eqref{eq:posfinpres} with $P$ a PFG projective profinite group, $R$ is PFG as a normal subgroup of $P$.
\end{defn}

We justify our use of the term positively finitely presented by showing that groups satisfying this condition are finitely presented.

\begin{prop}\label{prop:PFG and fin pres}
A profinite group is positively finitely presented if and only if it is PFG and finitely presented.
\end{prop}
\begin{proof}
Given a positively finitely presented group $G$, fix a presentation \eqref{eq:posfinpres} with $P$ PFG projective. Finitely generated projective profinite groups are finitely presented by \cite[Proposition 1.1]{Lubotzky}, so this exhibits $G$ as a quotient of a finitely presented group by a normally finitely generated group: it is standard that such groups are finitely presented.

PFG and finitely presented implies PFR by Proposition \ref{prop:PFGimpliesPFR}, and then PFR plus PFG imply positively finitely presented by \cite[Lemma 3.4]{KV}.
\end{proof}

In particular, PFG projective groups are positively finitely presented.

\begin{cor}
\label{pfp=pfr}
For PFG groups, positive finite presentation is equivalent to PFR.
\end{cor}
\begin{proof}
PFR groups are finitely presented, which with PFG implies positive finite presentation by Proposition \ref{prop:PFG and fin pres}. Positively finitely presented groups are PFG and finitely presented, which implies PFR by Proposition \ref{prop:PFGimpliesPFR}.
\end{proof}

The next two results show that the class of positively finitely presented profinite groups is well behaved.

\begin{lem}
\label{pfp-extension}
For $N$ a positively finitely presented normal subgroup of $G$, $G$ is positively finitely presented if and only if $G/N$ is.
\end{lem}
\begin{proof}
$N$ is finitely presented, so $G$ is finitely presented if and only if $G/N$ is; $N$ is PFG, so $G$ is PFG if and only if $G/N$ is. So if one of $G$ and $G/N$ is positively finitely presented, the other is finitely presented and PFG, so it is positively finitely presented by Proposition \ref{prop:PFG and fin pres}.
\end{proof}

Compare this to the class of PFR groups: it remains an open question whether this is closed under extensions (\cite[p.3]{KV}).

We conclude this section by showing that $G$ being positively finitely presented is witnessed by its universal Frattini cover. Compare this to the class of PFR groups: in general minimal presentations are not sufficient to determine whether a group is PFR (\cite[Section 7]{KV}). The following lemma is a generalisation of \cite[Proposition 11.2.1]{LS} to \emph{positively finitely normally generated subgroups}.

\begin{lem}
\label{lem:pfng}
Let $G$ be a profinite group, $B \lhd A \lhd G$ with $B$ normal in $G$. Suppose that $A/B$ is positively finitely normally generated in $G/B$ and $B$ is positively finitely normally generated in $G$. Then $A$ is positively finitely normally generated in $G$.
\end{lem}
\begin{proof}
We first consider the case with $A$ (and hence $B$) finite. Let $\pi: G^k \to (G/B)^k$ be the obvious projection, and pick $\underline{a} \in \pi^{-1}(X^{G/B}(A/B,k))$, $\underline{b}\in X^G(B,k)$ and $\underline{u} \in (\langle\underline{a} \rangle^G)^l$; then $\langle \underline{b}\cdot \underline{u},\underline{a}\rangle^G =A$ (the product $\underline{b}\cdot \underline{u}$ is componentwise). Thus we can estimate the probability of normally generating $A$ in $G$, by counting the possible choices for the elements $\underline{a}$, $\underline{b}$ and $\underline{u}$: there are at least $$ \vert B \vert^k \cdot \vert X^{G/B}(A/B,k) \vert \cdot \vert A/B \vert^{l} \cdot \vert X^G(B,l) \vert $$ choices of $k+l$ elements generating $A$, and we conclude that $P^G(A,k+l)$ is at least $$\frac{\vert B \vert^k  \vert X^{G/B}(A/B,k) \vert}{\vert A \vert^k} \frac{\vert A/B \vert^{l} \vert X^G(B,l) \vert}{\vert A\vert^l}= P^{G/B}(A/B,k) P^G(B,l).$$

Now suppose $A$ is profinite. Note that $A$ has a neighbourhood basis $\mathcal{N}$ of the identity consisting of open normal subgroups which are $G$-invariant: this can be achieved by intersecting a basis of open normal subgroups of $G$ with $A$.

For a subset $X\subset A$, $\langle X\rangle^G = A$ if and only if $\langle XN/N \rangle^{G/N}=A/N$ for all $N\in \mathcal{N}$ (this is \cite[Proposition 4.1.1]{Wilson} with minor modifications). This implies that $X^G(A,k)$ is the inverse limit over $\mathcal{N}$ of $X^{G/N}(A/N,k)$ and hence
\begin{align*}
P^G(A,k+l) &= \mu_{A^k}(X^G(A,k+l)) \\
&= \inf_{n\in \mathcal{N}} \frac{\vert X^{G/N}(A/N,k+l) \vert}{\vert A/N\vert^{k+l}} \\
&= \inf_{n\in \mathcal{N}} P^{G/N}(A/N,k+l) \\
&\geq \inf_{n\in \mathcal{N}} P^{G/BN}(A/BN,k) P^{G/N}(BN/N,l) \\
&= P^{G/B}(A/B,k) P^G(B,l),
\end{align*}
which is positive for some choice of $k$ and $l$ by hypothesis.
\end{proof}

\begin{prop}
\label{pfp-frattinicover}
Let $G$ be a PFG profinite group and let $f: \tilde{G}\to G$ be the universal Frattini cover of $G$. Write $R$ for the kernel of this map. If $R$ is positively normally finitely generated in $\tilde{G}$, $G$ is positively finitely presented.
\end{prop}
\begin{proof}
Since $\tilde{G}$ is a projective cover of $G$, if $1 \to S \to Q \to G \to 1$ is another presentation of $G$ with $Q$ PFG, then $Q \to G$ factors into an epimorphism $Q \to \tilde{G}$ and $f$. The diagram
\[\xymatrix{
S \ar[r] \ar@{->>}[d] & Q \ar[r] \ar@{->>}[d] & G \ar@{=}[d] \\
R \ar[r] & \tilde{G} \ar[r] & G
}\]
has exact rows; writing $T$ for the kernel of $Q \to \tilde{G}$, we get $S/T \cong R$ by the Nine Lemma. Since $R$ is positively finitely normally generated in $\tilde{G}$, it has polynomial maximal $\tilde{G}$-stable subgroup growth. $\tilde{G}$ is positively finitely generated and projective, hence positively finitely presented, so $T$ has polynomial maximal $Q$-stable subgroup growth. By Lemma \ref{lem:pfng} applied to $S$ as an extension of $T$ by $R$, we have that $S$ 
is positively normally finitely generated in $Q$, as required.
\end{proof}

\section{Modules of type \texorpdfstring{$\PFP_n$}{PFPn}}\label{sec:PFP_modules}

Let $R$ be a profinite ring. In this section, all modules will be profinite $R$-modules. 

%

\subsection{Projective covers of PFG modules}

Let $M$ be an $R$-module. A submodule $N$ of a module $M$ is \emph{superfluous} if, for any submodule $H$ of $M$, $H+N=M$ implies $H=M$. 

\begin{defn}
A homomorphism $P\to M$, with $P$ projective, is said to be a \emph{projective cover} of $M$ if its kernel is a superfluous submodule of $P$.
\end{defn}

It is easy to see that, if $P_1\to M$ and $P_2\to M$ are two projective covers of $M$, then $P_1\cong P_2$.  
So we may abuse terminology by referring to $P$ itself as the projective cover of $M$, instead of the homomorphism $P\to M$. Profinite modules have projective covers by \cite[Remark 3.4.3(i)]{SW}.

\begin{lem}
\label{projcover}
$M$ is PFG if and only if its projective cover $P$ is.
\end{lem}
\begin{proof}
This is the same argument as Lemma \ref{fratcover}: we can assume that both $M$ and $P$ are finitely generated and observe that, for any generating set $S$ of $M$, any lift of $S$ to $P$ generates $P$, since the kernel is superfluous.
\end{proof}

\subsection{Modules of type \texorpdfstring{$\PFP_n$}{PFPn}}\label{sec:PFPn_modules} The previous lemma suggests the following definition. 

\begin{defn}
An $R$-module $M$ has \emph{type $\PFP_n$} if it has a projective resolution $P_\ast$ 
\begin{equation*}
\ldots \to P_n \to \ldots \to P_1\to P_0 \to M \to 0 
\end{equation*}
with $P_0, \ldots, P_n$ PFG $R$-modules. The module $M$ has type $\PFP_\infty$ if it has a projective resolution $P_\ast$ with $P_n$ PFG for all $n$.
\end{defn}

\begin{prop}
\label{Schanuel}
Suppose we have two partial resolutions $$P_{n-1} \to \cdots \to P_0 \text{ and } Q_{n-1} \to \cdots \to Q_0$$ of $M$ with each $P_i$ and $Q_i$ PFG projective. Then $\ker(P_{n-1} \to P_{n-2})$ is PFG if and only if $\ker(Q_{n-1} \to Q_{n-2})$ is.
\end{prop}
\begin{proof}
Schanuel's lemma.
\end{proof}

It follows that $M$ has type $\PFP_\infty$ if and only if it has type $\PFP_n$ for all $n$.

\begin{rem}
\begin{enumerate}[(i)]
\item By Lemma \ref{projcover}, an $R$-module has type $\PFP_0$ if and only if it is PFG.
\item Clearly, type $\PFP_n$ implies type $\FP_n$ for all $n$.
\item Note that, if $R$ is PFG as an $R$-module, then all finitely generated $R$-modules are PFG. Thus, type $\PFP_n$ coincides with type $\FP_n$ for PFG rings.
\end{enumerate}
\end{rem}

We will now show that the properties defined above behave well with respect to short exact sequences. See \cite{Weibel} for more detail on the constructions used.

\begin{prop}
\label{moduletypes}
Let $0 \to A \xrightarrow[]{f} B \xrightarrow[]{g} C \to 0$ be a short exact sequence of profinite $R$-modules.
\begin{enumerate}[(i)]
\item If $A$ has type $\PFP_{n-1}$ and $B$ has type $\PFP_n$, $C$ has type $\PFP_n$.
\item If $B$ has type $\PFP_{n-1}$ and $C$ has type $\PFP_n$, $A$ has type $\PFP_{n-1}$.
\item If $A$ and $C$ have type $\PFP_n$, so does $B$.
\end{enumerate}
\end{prop}
\begin{proof}
Recall that the class of PFG modules is closed under extensions by Lemma \ref{lem:PFGquotext}.
\begin{enumerate}[(i)]
\item Take a type $\PFP_{n-1}$ resolution $P'_\ast$ of $A$ and a type $\PFP_n$ resolution $P_\ast$ of $B$. There is a map $P'_\ast \to P_\ast$ extending $A \to B$. The mapping cone of $P'_\ast \to P_\ast$ is a type $\PFP_n$ resolution of $C$.
\item Fix a map $Q \stackrel{q}{\twoheadrightarrow} B$ with $Q$ PFG projective. Note that $Q$ has type $\PFP_\infty$. We have a diagram
\[
\xymatrix{
	\ker(q) \ar[r] \ar[d] & Q \ar@{->>}[r]^{q} \ar@{=}[d] & B \ar@{->>}[d]^{g} \\
	\ker(g \circ q) \ar[r] & Q \ar@{->>}[r] & C;}
\]
with exact rows. By Proposition \ref{Schanuel}, $\ker(q)$ is of type $\PFP_{n-2}$ and $\ker(g\circ q)$ is of type $\PFP_{n-1}$. By the snake lemma, $0 \to \ker(q) \to \ker(g\circ q) \to A \to 0$ is exact. By (i), $A$ has type $\PFP_{n-1}$.
\item Given a type $\PFP_n$ resolution for $A$ and another for $C$, the resolution for $B$ constructed using the horseshoe lemma has type $\PFP_n$.
\end{enumerate}
\end{proof}

\subsection{Growth conditions}

The famous Mann-Shalev Theorem in \cite{MS} characterises PFG groups algebraically as those profinite groups with ``few'' open maximal subgroups. We would like to mimic this theorem as well as producing a cohomological criterion for when modules have type $\PFP_n$.

\begin{defn}
For $R$ a profinite ring and $M$ a profinite $R$-module, let $m^R_k(M)$ be the number of maximal (open) submodules of $M$ of index $k$. We say that $M$ has \emph{polynomial maximal submodule growth}, or \emph{PMSMG} for short, if there is some constant $c>0$ such that $m^R_k(M) \leq k^c$ for all $k$.
\end{defn}

For a profinite ring $R$ and $k\in \mathbb{N}$, denote by $\mathcal{S}^R_k$ the set of irreducible $R$-modules of order $k$.

\begin{prop}\label{prop:PMSMG}
Let $M$ be a profinite $R$-module. Then the following conditions are equivalent:
\begin{enumerate}[(i)]
\item $M$ is PFG.
\item $M$ has PMSMG.
\item $\sum_{S \in \mathcal{S}^R_k} (\lvert \Hom_R(M,S) \rvert-1)$ has polynomial growth in $k$.
\end{enumerate} 
\end{prop}
\begin{proof}
(i)$\Leftrightarrow$ (ii): Imitate \cite[Proposition 6.1, (2)$\Rightarrow$(3)]{KV} and \cite[Proposition 3.5]{KV}.

We will now show that (ii) is equivalent to (iii). First, note that for each maximal submodule in $M$ we get a quotient map to an irreducible module, so we have an injection from the set of maximal submodules of index $k$ to the set of surjective (or equivalently, non-trivial) maps to irreducible modules of order $k$. Hence, $$m^R_k(M) \leq \sum_{S \in \mathcal{S}^R_k} (\lvert \Hom_R(M,S) \rvert-1).$$

Conversely, if $M$ has PMSMG, we have shown it is $d$-generated for some $d$. Hence $\lvert \Hom_R(M,S) \rvert \leq \lvert S \rvert^d$, so $$\sum_{S \in \mathcal{S}^R_k} (\lvert \Hom_R(M,S) \rvert-1) \leq k^dm^R_k(M).$$
\end{proof}

Since $\mathrm{Hom}(M,S)$ is just $\Ext^0_R(M,S)$, we can now apply the proposition to give conditions equivalent to a module having type $\PFP_n$.

\begin{thm}
\label{extcondition}
Let $M$ be an $R$-module. Then, $M$ has type $\PFP_n$ if and only if $\sum_{S \in \mathcal{S}^R_k} (\lvert \Ext_R^m(M,S) \rvert-1)$ has polynomial growth in $k$ for all $m \leq n$.
\end{thm}
\begin{proof}
For an $R$-module $M$, we write $$f_m^M(k)= \sum_{S \in \mathcal{S}^R_k} (\lvert \Ext_R^m(M,S) \rvert-1).$$

The case $n=0$ is Proposition \ref{prop:PMSMG}. Now, suppose that $n\ge 1$ and the theorem is true for every $m\le n-1$. Let $M$ be an $R$-module of type $\PFP_n$. By hypothesis, $f_m^M(k)$ is polynomial in $k$ for $m \leq n-1$; it remains to check that $f_n^M(k)$ is polynomial in $k$. By Lemma \ref{projcover}, we have a short exact sequence $$0 \to K \to P \to M \to 0$$ with $P$ PFG projective. $K$ has type $\PFP_{n-1}$, so by hypothesis $f_{n-1}^K(k)$ has polynomial growth in $k$; since $\Ext_R^n(P,-)=0$ for $n \geq 1$, we have that $f_n^P(k)=0$. Using the long exact sequence in cohomology, we see that $f_n^M(k) \leq f_{n-1}^K(k) + f_n^P(k)$, and we are done.
\end{proof}

Note that $\mathcal{S}^R_k$ may be infinite, and the sum $\sum_{S \in \mathcal{S}^R_k} (\lvert \Ext_R^n(M,S) \rvert-1)$ may nonetheless be finite. We will see in Section \ref{examples} that, for an infinite product $G=\prod H$ with $H$ a finite group,, there are infinitely many values of $k$ such that $$\sum_{S \in \mathcal{S}^{\hat{\mathbb{Z}}\llbracket G \rrbracket}_k} (\lvert H_{\hat{\mathbb{Z}}}^1(G,S) \rvert-1) = 0$$ even though $\mathcal{S}^{\hat{\mathbb{Z}}\llbracket G \rrbracket}_k$ is infinite.

\section{Groups of type \texorpdfstring{$\PFP_n$}{PFPn}}
\label{pfpn_groups}

In this section $R$ will be a commutative profinite ring.

\begin{defn}
A profinite group $G$ has \emph{type $\PFP_n$ over $R$} if $R$ has type $\PFP_n$ as $R \llbracket G \rrbracket$-module. 
Unless specified otherwise, type $\PFP_n$ will mean over $\hat{\mathbb{Z}}$.
\end{defn}

\begin{rem}[$\PFP_n$ for prosoluble groups]\label{ex:PFPnprosoluble}
We get immediately that $\PFP_n$ and $\FP_n$ are equivalent when $R \llbracket G \rrbracket$ itself is PFG as an $R \llbracket G \rrbracket$-module. By Corollary \ref{PFGtoUBERG}, the ring $\hat{\mathbb{Z}} \llbracket G \rrbracket$ is PFG as a $\hat{\mathbb{Z}} \llbracket G \rrbracket$-module whenever $G$ is PFG; by Remark \ref{rem:prop6.1notfingen}, this occurs for all finitely generated prosoluble groups $G$. So $\PFP_n$ and $\FP_n$ over $\hat{\mathbb{Z}}$ coincide for such groups (cf.\ Conjecture \ref{quest:prosoluble} below).
\end{rem}

Every group is of type $\FP_0$ over every ring. This is false for type $\PFP_0$.
 
\begin{lem}
A profinite group $G$ has type $\PFP_0$ over $R$ if and only if $R$ is PFG as an $R$-module.
\end{lem}
\begin{proof}
Any subset of $R$ generates it as an $R$-module if and only if it generates it as an $R \llbracket G \rrbracket$-module, since the $G$-action is trivial.
\end{proof}

Next we show that the class of groups of type $\PFP_n$ is closed under commensurability. Recall that for an $R \llbracket G \rrbracket$-module $M$, we denote by $\Res^G_H M$ the $ R \llbracket H \rrbracket$-module with the same underlying set as $M$ and restricting the action of $G$ to $H$.

\begin{prop}\label{prop:PFPncommensurability}
Let $G$ be a profinite group and let $H$ be an open subgroup of $G$. Then, $H$ has type $\PFP_n$ over $R$ if and only if $G$ does.
\end{prop}
\begin{proof}
We first claim that an $R \llbracket G \rrbracket$-module $M$ is PFG if and only if $\Res^G_H M$ is PFG. Indeed, clearly any set of generators for $\Res^G_H M$ as an $R \llbracket H \rrbracket$-module generates $M$ as an $R \llbracket G \rrbracket$-module. Conversely, say $\lvert G:H \rvert = c$. A set of $t$ generators for $M$ generates an open submodule of $\Res^G_HM$ of index at most $c^t$, so if $P_{R \llbracket G \rrbracket}(M,t) > 0$ for some $t$, $P_{R \llbracket H \rrbracket}(\Res^G_HM,t+c^t) > 0$.

It now follows by the same techniques as for abstract modules (see \cite[VIII, Proposition 5.1]{Brown}) that $M$ has type $\PFP_n$ if and only if $\Res^G_HM$ does, and in particular this holds for $M=R$.
\end{proof}

\subsection{Minimal resolutions}

We can imitate the methods of \cite{Damian} to give some more detail about the type $\PFP_n$ conditions. Fix a projective resolution $P^G_\ast$ of $\hat{\mathbb{Z}}$ as a $\hat{\mathbb{Z}}\llbracket G \rrbracket$-module such that each $P^G_n$ is a projective cover of the kernel $K^G_{n-1}$ of the following map. Thus $G$ has type $\PFP_n$ over $\hat{\mathbb{Z}}$ if and only if $K^G_m$ is PFG for all $m<n$.

\begin{define}
Write $\mathcal{S}^{\hat{\mathbb{Z}}\llbracket G \rrbracket}$ for the set of irreducible $\hat{\mathbb{Z}}\llbracket G \rrbracket$-modules. For $M \in \mathcal{S}^{\hat{\mathbb{Z}}\llbracket G \rrbracket}$:
\begin{enumerate}[(i)]
\item $r_G(M)$ is the dimension of $M$ over the field $\End_G(M)$;
\item $\delta_G(M)$ is the number of non-Frattini chief factors $G$-isomorphic to $M$ in a chief series of $G$;
\item $h_G^n(M)$ is the dimension of $H^n(G,M)$ over $\End_G(M)$, for $n \geq 1$;
\item $h'_G(M)$ is the dimension of $H^1(G/C_G(M),M)$ over $\End_G(M)$.
\end{enumerate}
For $N$ any other (profinite) $\hat{\mathbb{Z}}\llbracket G \rrbracket$-module, we write $i_N(M)$ for the number of factors $G$-isomorphic to $M$ in the (semisimple) \emph{head} $N/\operatorname{rad}(N)$ of $N$. By \cite[Proposition~7.4.5]{Wilson}, $N/\operatorname{rad}(N)$ is isomorphic to a product of simple modules and we can write $N/\operatorname{rad}(N) = \prod_{i\in I} H_i$, where $H_i$ is a power of a simple $\hat{\mathbb{Z}}\llbracket G \rrbracket$-module $M_i$ and $M_i \not\cong M_j$ for $i\neq j$. The profinite $\hat{\mathbb{Z}}\llbracket G \rrbracket$-module $H_i$ is called a homogeneous component of the head of $N$.
\end{define}

Parallel to \cite[Theorem 1]{Damian}, we get:

\begin{thm}
For $N$ a $\hat{\mathbb{Z}}\llbracket G \rrbracket$-module, the minimal size $d_{\hat{\mathbb{Z}}\llbracket G \rrbracket}(N)$ of a generating set of $N$ is $$\sup_{M \in \mathcal{S}^{\hat{\mathbb{Z}}\llbracket G \rrbracket}} \left \lceil \frac{i_N(M)}{r_G(M)} \right\rceil.$$
\end{thm}
\begin{proof}
Any subset of $N$ generates $N$ if and only if it generates the head of $N$, if and only its projection into each homogeneous component of the head generates that homogeneous component (see \cite[Section 3]{Damian}). Each homogeneous component is a product of $i_N(M)$ copies of some $M \in \mathcal{S}^{\hat{\mathbb{Z}}\llbracket G \rrbracket}$, and so the homogeneous component is generated by $\left \lceil i_N(M)/r_G(M) \right\rceil$ elements by \cite[Lemma 1]{CGK}.
\end{proof}

We can apply this theorem to the kernels $K^G_n$ in our minimal projective resolution.

\begin{cor}
$$d_{\hat{\mathbb{Z}}\llbracket G \rrbracket}(K^G_{n-1}) = \sup_{M \in \mathcal{S}^{\hat{\mathbb{Z}}\llbracket G \rrbracket}} \left\lceil \frac{h_G^n(M)}{r_G(M)} \right\rceil.$$
\end{cor}
\begin{proof}
$i_{K^G_{n-1}}(M) = h_G^n(M)$ by the argument of \cite[Lemma 2.11]{Gruenberg}: differentials in the chain complex $\Hom_G(P^G_\ast,M)$ are trivial for $M$ irreducible.
\end{proof}

\begin{cor}
$$d_{\hat{\mathbb{Z}}\llbracket G \rrbracket}(K^G_0) = \sup_{M \in \mathcal{S}^{\hat{\mathbb{Z}}\llbracket G \rrbracket}} \left \lceil \frac{\delta_G(M)+h'_G(M)}{r_G(M)} \right \rceil.$$
\end{cor}
\begin{proof}
$h_G^1(M) = \delta_G(M)+h'_G(M)$, by \cite[(2.10)]{AG}.
\end{proof}


Note that it is enough for the two corollaries to take the supremum over $\mathbb{F}_p$-modules for $p$ dividing the order of $G$, since otherwise $h_G^n(M)=0$ by \cite[Corollary 7.3.3]{RZ}.

As expected, we deduce that $K^G_0$ is finitely generated if and only if $I_{\hat{\mathbb{Z}}}\llbracket G \rrbracket$ is, and comparing our theorem to \cite[Theorem 1]{Damian}, we see that $$d_{\hat{\mathbb{Z}}\llbracket G \rrbracket}(K^G_0) \leq d_{\hat{\mathbb{Z}}\llbracket G \rrbracket}(I_{\hat{\mathbb{Z}}}\llbracket G \rrbracket) \leq d_{\hat{\mathbb{Z}}\llbracket G \rrbracket}(K^G_0)+1.$$

\begin{thm}
For $N$ a $\hat{\mathbb{Z}}\llbracket G \rrbracket$-module, and any $k$, $$P_{\hat{\mathbb{Z}}\llbracket G \rrbracket}(N,k) = \prod_{M\in \mathcal{S}^{\hat{\mathbb{Z}}\llbracket G \rrbracket}} \prod_{i=0}^{i_N(M)-1} \left(1-\frac{\vert \End_G(M) \vert^i}{\vert M \vert^k}\right).$$ (If $i_N(M)=0$ for some $M$, we take the corresponding factor to be the empty product, i.e. $1$.)
\end{thm}
\begin{proof}
The proof echoes that of \cite[Theorem 2]{Damian} exactly, except for the following two points.

Firstly, it is unnecessary to assume $k \geq d_{\hat{\mathbb{Z}}\llbracket G \rrbracket}(N)$: if not, by the argument of the previous theorem there is some $M \in \mathcal{S}^{\hat{\mathbb{Z}}\llbracket G \rrbracket}$ such that $$i_N(M)/r_G(M) > k,$$ so there is some $i < i_N(M)$ such that $$\vert M \vert^k = \vert \End_G(M) \vert^{r_G(M)k} = \vert \End_G(M) \vert^i,$$ and hence for that value of $i$, the factor $1-\vert \End_G(M) \vert^i/\vert M \vert^k$ vanishes, and we correctly conclude $P_{\hat{\mathbb{Z}}\llbracket G \rrbracket}(N,k) = 0$. If $N$ is not finitely generated, our formula gives $P_{\hat{\mathbb{Z}}\llbracket G \rrbracket}(N,k) = 0$ for all $k$.

Secondly, it is unnecessary to assume $G$ is countably based. If not, we may choose a (transfinite) sequence $\{N_\alpha\}$ of normal closed subgroups of $G=N_0$, such that $N_\alpha+1$ is open in $N_\alpha$ and $N_\alpha = \bigcap_{\beta<\alpha} N_\beta$ for $\alpha$ a limit. Then, arguing inductively on $\alpha$ in the same way as \cite[Theorem 2]{Damian}, we reach the required conclusion.
\end{proof}

These arguments apply to the statement of \cite[Theorem 2]{Damian}: the extra conditions there can be removed.

\begin{cor}
For any $k$, $$P_{\hat{\mathbb{Z}}\llbracket G \rrbracket}(K^G_{n-1},k) = \prod_{\mathcal{S}^{\hat{\mathbb{Z}}\llbracket G \rrbracket}} \prod_{i=0}^{h_G^n(M)-1} \left(1-\frac{\vert \End_G(M) \vert^i}{\vert M \vert^k}\right).$$
\end{cor}
\begin{proof}
As before, note that $i_{K^G_{n-1}}(M) = h_G^n(M)$.
\end{proof}

Finally, by applying the same arguments as \cite[Theorem 3]{Damian}, we get:

\begin{thm}
For $N$ a finitely generated $\hat{\mathbb{Z}}\llbracket G \rrbracket$-module, $N$ is PFG if and only if the number of irreducible $G$-modules $M$ such that $i_N(M) \geq 1$ is polynomially bounded as a function of the order.
\end{thm}

Note that the inequality $\sum g(n)/n^t \leq A_G(t)$ in the proof of \cite[Theorem 3]{Damian} (in the notation used there) is not quite true: it assumes that $i_G(M) \geq 1$ for all $M \in \mathcal{S}^{\hat{\mathbb{Z}}\llbracket G \rrbracket}$. Since this holds for all non-trivial $M$, and since the number of trivial irreducible $G$-modules of order $n$ is polynomially bounded as a function of $n$ (by the polynomial $1$), the conclusion holds.

\begin{cor}
\label{Damian3}
Assume $G$ has type $\FP_n$. Then $K^G_{n-1}$ is PFG (and $G$ has type $\PFP_n$ over $\hat{\mathbb{Z}}$) if and only if the number of irreducible $G$-modules $M$ such that $h_G^n(M) \geq 1$ is polynomially bounded as a function of the order. In particular, for $G$ of type $\FP_1$, $G$ has type $\PFP_1$ if and only if the number of irreducible $G$-modules $M$ such that $\delta_G(M)+h'_G(M) \geq 1$ is polynomially bounded as a function of the order.
\end{cor}

\subsection{APFG and \texorpdfstring{$\mathrm{PFP}_1$}{PFP1}}\label{sec:APFG}

As the reader might guess, APFG and $\PFP_1$ are closely related. From Proposition \ref{APFGring} we deduce:

\begin{lem}
\label{APFGPFP1}
If $G$ is APFG, then it has type $\PFP_1$ over $\hat{\mathbb{Z}}$.
\end{lem}
\begin{proof}
By Proposition \ref{APFGring}, $\hat{\mathbb{Z}} \llbracket G \rrbracket$ is PFG. So the exact sequence \eqref{eq:augmentation} of PFG modules shows that $G$ has type $\PFP_1$.
\end{proof}

\begin{cor}
\label{PFP1}
If $G$ is PFG, then it has type $\PFP_1$ over $\hat{\mathbb{Z}}$.
\end{cor}
\begin{proof}
By Proposition \ref{prop:damian_thm4.4}, $G$ being PFG implies it is APFG. The result follows by the previous lemma.
\end{proof}

Using the relation between APFG and $\PFP_1$, we can give an example of group of type $\PFP_1$ that is not PFG.

\begin{example}\label{ex:damian}
In \cite[Example 4.5]{Damian} it is shown that, for $N$ large enough, the group $\prod_{n\ge N} \mathrm{Alt}(n)^{2^n}$ is APFG but not PFG. This example therefore has type $\PFP_1$ over $\hat{\mathbb{Z}}$, and has UBERG, but is not PFG.
\end{example}

\subsection{Minimal extensions, PFR and type \texorpdfstring{$\PFP_2$}{PFP2}}
\label{PFPvsPFP2}


\begin{lem}
\label{pfp2}
PFR implies $\PFP_2$ over $\hat{\mathbb{Z}}$.
\end{lem}
\begin{proof}
Let $G$ be PFR. Then $G$ is finitely presented and has UBERG, by Proposition \ref{KV}. Hence $G$ has type $\PFP_2$ if and only if it has type $\FP_2$. Since $G$ is finitely presented, it has type $\FP_2$.
\end{proof}

Recall that group extensions $K \to E \to G$ and $K \to E' \to G$ with $K$ abelian are said to be \emph{equivalent} if there is a commutative diagram
\[
\xymatrix{
K \ar[r] \ar@{=}[d] & E \ar[r] \ar[d] & G \ar@{=}[d] \\
K \ar[r] & E' \ar[r] & G;}
\]
they are said to be \emph{isomorphic} if there is a commutative diagram
\[
\xymatrix{
K \ar[r] \ar^\cong[d] & E \ar[r] \ar[d] & G \ar@{=}[d] \\
K \ar[r] & E' \ar[r] & G.}
\]
So equivalent abelian extensions are isomorphic. Moreover, isomorphic abelian extensions induce a $G$-automorphism of $K$; since irreducible $G$-modules are $1$-generated, the number of $G$-automorphisms is at most $\lvert K \rvert$. Therefore the number of equivalence classes of minimal abelian extensions of degree $n$ in one isomorphism class is at most  $n$.

We conclude:

\begin{prop}
\label{eq/iso}
Polynomial minimal abelian extension growth (that is, polynomial growth in the number of isomorphism classes of minimal extensions of $G$ with abelian kernel) is equivalent to polynomial growth in the number of equivalence classes of minimal extensions by non-isomorphic $G$-modules.
\end{prop}

Recall that, by the usual correspondence between the second cohomology of a group and its extensions, this second condition is equivalent to saying that $\sum_S \vert H^2(G,S) \vert$ grows polynomially in $k$, where the sum ranges over all irreducible $G$-modules of order $k$.

We can use this idea for an alternative proof that PFR implies type $\PFP_2$. It is shown in \cite[Theorem 3.9]{KV} that for finitely generated profinite groups, PFR is equivalent to PMEG (see Section \ref{sec:pfgpfrmore}) -- that is, polynomial growth in the number of isomorphism classes of minimal extensions. Clearly this implies polynomial minimal abelian extension growth, and in fact it is equivalent to it among finitely presented groups, by \cite[Theorem 4.1, Proposition 5.2, Proposition 5.4]{KV}. By our proposition, this is equivalent to polynomial growth in $\sum_S \vert H^2(G,S) \vert$; by Theorem \ref{extcondition}, groups of type $\PFP_1$ (which includes PFR groups by Proposition \ref{APFGring} and Lemma \ref{APFGPFP1}) have type $\PFP_2$ (over $\hat{\mathbb{Z}}$) if and only if $\sum_S (\vert H^2(G,S) \vert-1)$ grows polynomially, so the result follows.

We also observe, in contrast to \cite[Section 7]{KV}, that we can check the $\PFP_2$ property by considering the minimal presentation of a group.

Imitating \cite[Section 3]{KV}, we say that a presentation \eqref{eq:posfinpres} of $G$ has polynomial maximal $P$-stable \emph{abelian} subgroup growth if the number of maximal $P$-stable subgroups $S$ of $R$ of index $n$ with $R/S$ abelian grows polynomially in $n$.

We can now state an equivalent formulation of property $\PFP_2$.

\begin{prop}
\label{PFP2=absubgps}
Let $G$ be a profinite group of type $\PFP_1$ and let $\tilde{G}\to G$ be the universal Frattini cover of $G$. Then $G$ has type $\PFP_2$ if and only if $R = \ker(\tilde{G} \to G)$ has polynomial maximal $\tilde{G}$-stable abelian subgroup growth. 
\end{prop}
\begin{proof}
We may use the argument of \cite[Proposition 7.1]{KV}: the maximal $\tilde{G}$-stable subgroups of $R$ of degree $n$ correspond precisely to the isomorphism classes of non-split minimal abelian extensions of $G$ of degree $n$ by \cite[(3.2)]{Hill}; the number of these grows polynomially in $n$ if and only if the number of equivalence classes of non-split minimal abelian extensions of $G$ of degree $n$ does, just as in Proposition \ref{eq/iso}; since $G$ has type $\PFP_1$, this condition is equivalent to having type $\PFP_2$ by Theorem \ref{extcondition}.
\end{proof}

\subsection{Relative type \texorpdfstring{$\PFP_n$}{PFPn}}
\label{sec:relpfpn}

\begin{defn}
A closed normal subgroup $H$ of a profinite group $G$ has \emph{relative type $\PFP_n$ in $G$ over $R$} if all PFG projective $R \llbracket G/H \rrbracket$-modules have type $\PFP_n$ over $R \llbracket G \rrbracket$.
\end{defn}

 \begin{rem}
 Notice that in the analogous $\FP_n$ case, $H$ has relative type $\FP_n$ in all groups if and only if $H$ has type $\FP_n$, whereas when $G$ does not have type $\PFP_1$, the trivial subgroup (of type $\PFP_\infty$) does not have relative type $\PFP_1$ in $G$, so we need the two different definitions.
\end{rem}

\begin{thm}\label{relPFPn}
Let $G$ be a profinite group, $H$ be normal in $G$ and let $M$ be a profinite $R \llbracket G/H \rrbracket$-module. Suppose $H$ has relative $\PFP_m$ in $G$ over $R$.
\begin{enumerate}[(i)]
\item If $M$ has type $\PFP_n$ over $R \llbracket G \rrbracket$ (via restriction), then it has type $\PFP_{k}$ over $R \llbracket G/H \rrbracket$, where $k = \min(m+1,n)$.
\item If $M$ has type $\PFP_n$ over $R \llbracket G/H \rrbracket$, then it has type $\PFP_{l}$ over $R \llbracket G \rrbracket$, where $l = \min(m,n)$.
\end{enumerate}
\end{thm}
\begin{proof}
First, note that $M$ is PFG as an $R \llbracket G \rrbracket$-module if and only if it is PFG as an $R \llbracket G/H \rrbracket$-module, because the action is by restriction. 
\begin{enumerate}[(i)]
\item Use induction on $n$. When $n=0$ we are done. Take a PFG projective $R \llbracket G/H \rrbracket$-module $P$ and an epimorphism $P \to M$ with kernel $K$. The module $K$ has type $\PFP_{\min(m,n-1)}$ over $R \llbracket G \rrbracket$ by Proposition \ref{moduletypes}, so by hypothesis it has type $$\min(m+1,\min(m,n-1)) = \min(m,n-1)$$ over $R \llbracket G/H \rrbracket$. Therefore $M$ has type $\min(m,n-1)+1=k$ over $R \llbracket G/H \rrbracket$.
\item Use induction on $n$. When $n=0$ we are done. Take a PFG projective $R \llbracket G/H \rrbracket$-module $P$ and an epimorphism $P \to M$ with kernel $K$. Now $K$ has type $\PFP_{n-1}$ over $R \llbracket G/H \rrbracket$-module by Proposition \ref{moduletypes}, so by hypothesis it has type $\PFP_{\text{min}(m,n-1)}$ over $R \llbracket G \rrbracket$. Also $P$ has type $\PFP_m$ over $R \llbracket G \rrbracket$, so by Proposition \ref{moduletypes} $M$ has type $\PFP_{\text{min}(m,n)}$ over $R \llbracket G \rrbracket$.
\end{enumerate}
\end{proof}

In particular this holds for $M=R$.

\begin{cor}
\label{relPFPn2}
Suppose that $H$ has relative type $\PFP_m$ in $G$ over $R$.
\begin{enumerate}[(i)]
\item If $G$ has type $\PFP_n$ over $R$, then $G/H$ has type $\PFP_{\min(m+1,n)}$ over $R$.
\item If $G/H$ has type $\PFP_n$ over $R$, then $G$ has type $\PFP_{\min(m,n)}$ over $R$.
\end{enumerate}
\end{cor}

\begin{rem}
Exactly the same approach gives the analogous classical result (see \cite[Proposition 2.7]{Bieri} for example) that relates type $\FP_n$ conditions for extensions and quotients of abstract groups. As far as we know, this is the first proof of this result which avoids the use of spectral sequences; those unfamiliar with the mysteries of spectral sequences may find this new perspective enlightening.
\end{rem}

At the moment the property of relative type $\PFP_n$, and thus the behaviour of type $\PFP_n$ under extensions, remains mysterious. But we have the following result:

\begin{prop}\label{prop:PFPndirectprod}
Suppose $G$ and $H$ are profinite groups. Let $M$ be an $R \llbracket G \rrbracket$-module and $N$ an $R \llbracket H \rrbracket$-module. Suppose $M$ has type $\PFP_m$ and $N$ has type $\PFP_n$. Then $M \otimes_R N$ has type $\PFP_{\min(m,n)}$ as an $R \llbracket G \times H \rrbracket$-module.
\end{prop}
\begin{proof}
We will show that if $P$ is a PFG projective $R \llbracket G \rrbracket$-module and $Q$ a PFG projective $R \llbracket H \rrbracket$-module, then $P \otimes_R Q$ is PFG as an $R \llbracket G \times H \rrbracket$-module. The result follows by taking tensor products of partial PFG projective resolutions for $M$ and $N$.


Each maximal submodule of $P$ gives an epimorphism onto an irreducble $k\llbracket G \rrbracket$-module for some field $k$ which is a quotient of $R$. Now define the function $f^G_n(P,k')$ over all finite extensions $k'$ of fields $k$ which appear as quotients of $R$, as follows. Let $\mathcal{S}_n(G,k')$ be the set of absolutely irreducible representations of $G$ of dimension $n$ which are defined over $k'$; we think of elements of $\mathcal{S}_n(G,k')$ as $R \llbracket G \rrbracket$-modules via restriction along $R \to k'$. \cite[Lemma 6.7]{KV} gives a bijection $\Phi_k$ from Galois orbits of irreducible $\bar{k}\llbracket G \rrbracket$-modules to irreducible $k\llbracket G \rrbracket$-modules, where $\bar{k}$ is the algebraic closure of $k$, and we identify $k'$ with a subfield of $\bar{k}$. Then $f^G_n(P,k') = \sum_{M \in \mathcal{S}_n(G,k')} (|\Hom_{R \llbracket G \rrbracket}(P,\Phi_k(M))|-1)$. Exactly the same approach as \cite[Lemma 6.8]{KV} shows that $P$ is PFG if and only if there is some $b$ such that $f^G_n(P,k') \leq |k'|^{bn}$ for all $n$ and all $k'$ where it is defined.

As in the proof of \cite[Theorem 6.4]{KV}, the absolutely irreducible representations of $G \times H$ over $k'$ are precisely the tensor products of absolutely irreducible representations of $G$ and $H$ over $k'$. As there, we deduce that if $f^G_n(P,k') \leq |k'|^{bn}$ for some $b$, and similarly for $f^H_n(Q,k')$, there exists a $c$ such that $f^{G \times H}_n(P \otimes_R Q,k') \leq |k'|^{cn}$, and $P \otimes_R Q$ is PFG, as required.

\end{proof}

\begin{cor}
If $G$ has type $\PFP_m$ over $R$ and $H$ has type $\PFP_n$ over $R$, then $G \times H$ has type $\PFP_{\min(m,n)}$ over $R$.
\end{cor}

Compare this to the result in \cite[Theorem 6.4]{KV} that UBERG is preserved by (finite) direct products.

\section{Examples}
\label{examples}


In this section we will construct some examples of groups of type $\PFP_1$ over $R$ although the group ring $R \llbracket G \rrbracket$ is not PFG, examples of groups of type $\FP_1$ over $\hat{\mathbb{Z}}$ which do not have type $\PFP_1$, and examples of groups of type $\PFP_n$ but not $\PFP_{n+1}$ over $\hat{\mathbb{Z}}$ for all $n$.

\subsection{\texorpdfstring{$G$ of type $\PFP_1$ over $R$ such that $R \llbracket G \rrbracket$ is not PFG}{G PFP1 != R[[G]] PFG}} The examples studied in \cite{Damian} suggest the strategy of looking at products of finite groups.

\begin{prop}\label{prop:finsimplegroup}
Let $S$ be a finite group and let $G$ be an infinite product of copies of $S$. For any prime $p$ not dividing the order of $S$, $\mathbb{Z}_p \llbracket G \rrbracket$ is not PFG, but $G$ has type $\PFP_1$ over $\mathbb{Z}_p$.
\end{prop}
\begin{proof}
The group $G$ has infinitely many irreducible representations of dimension $\leq \vert S\vert$ over $\mathbb{F}_p$, so $\mathbb{Z}_p \llbracket G \rrbracket$ is not PFG. 

On the other hand, \cite[Corollary 7.3.3]{RZ} shows that $H^1(G,M) =0$ is trivial for all irreducible $\mathbb{F}_p\llbracket G\rrbracket$-modules $M$ with $p \notin \pi$. Therefore, by Theorem \ref{extcondition}, we can see that $G$ has type $\PFP_1$ over $\mathbb{Z}_p$. 
\end{proof}

\begin{rem}
Similarly, if we denote by $\pi= \pi(S)$ the set of prime divisors of the order of $S$, the group $S^\mathbb{N}$ has type $\PFP_1$ over $\mathbb{Z}_{\pi'} = \prod_{p\notin \pi} \mathbb{Z}_p$.
\end{rem}

By varying the group $S$, we get examples of this behaviour over $\mathbb{Z}_p$ for all primes $p$. We do not know of any such examples over $\hat{\mathbb{Z}}$, and leave it as a question.

\begin{question}
Are there groups of type $\PFP_1$ over $\hat{\mathbb{Z}}$ which do not have UBERG?
\end{question}

By Corollary \ref{Damian3}, a necessary and sufficient condition for this is that the number of irreducible modules $M$ of order $n$ for such a group $G$ would grow faster than polynomially in $n$, but the number for which $\delta_G(M)+h'_G(M)$ is non-zero would grow polynomially.

\cite{Damian} constructs groups with UBERG which are not PFG, but another remaining question is whether such groups must be finitely generated.

\begin{question}
Are there groups with UBERG which are not finitely generated?
\end{question}
 
On the other hand, examples like the above cannot appear among pro\-nil\-po\-tent groups. In future work with S. Kionke, we show:
 
\begin{prop}\label{prop:pronilp}
Let $G$ be a pronilpotent group. Then $G$ has UBERG if and only if $G$ is finitely generated.
\end{prop}

The class of prosoluble groups appears often in these contexts as groups where pathological behaviour cannot occur. For example, fi\-ni\-te\-ly generated prosoluble groups are PFG, and prosoluble groups have type $\PFP_1$ if and only if they are finitely generated by Corollary \ref{PFP1} and \cite[Remark 3.5(a)]{CC2}. So it is very natural to ask:

\begin{question}\label{quest:prosoluble}
Are all prosoluble groups with UBERG finitely generated?
\end{question}

\subsection{Type \texorpdfstring{$\FP_1$}{FP1} and type \texorpdfstring{$\PFP_1$}{PFP1}}
\label{FP1vsPFP1}

Clearly type $\PFP_1$ over $R$ implies type $\FP_1$ over $R$; we show the converse does not hold.

\begin{prop}\label{prop:free_not_PFP}
Let $F_n$ be the free profinite group on $n$ generators. For $n > 1$, $F_n$ does not have type $\PFP_1$ over $\hat{\mathbb{Z}}$.
\end{prop}
\begin{proof}
We first give a proof for $n=3$; the case for $n>3$ is proved similarly.

We think of $F_3$ as the profinite free product $\hat{\mathbb{Z}} \ast F_2$, for some fixed copy of $\hat{\mathbb{Z}}$. We use the Mayer-Vietoris sequence of \cite[Proposition 9.2.13]{RZ}. Note that our profinite free product is proper by \cite[Example 9.2.6]{RZ}, so the Mayer-Vietoris sequence applies.

Now $F_3$ has type $\FP_1$ because it is finitely generated. To show it does not have type $\PFP_1$ we use the cohomological characterisation: we will show $\sum_{S \in \mathcal{S}^{F_3}_k} |H^1(G,S)|-1$ grows faster than polynomially in $k$. By the Mayer-Vietoris sequence, it suffices to show $\sum_{S \in \mathcal{S}^{F_3}_k} |H^1(\hat{\mathbb{Z}},S)|-1$ grows faster than polynomially in $k$.

Let $\mathcal{T}^{F_3}_k$ be the set of irreducible $F_3$-modules of order $k$ on which restriction to $\hat{\mathbb{Z}}$ gives the trivial action. By the universal property of free products, we can identify this with $\mathcal{S}^{F_2}_k$. The sequence $|\mathcal{S}^{F_2}_k|$ grows faster than polynomially in $k$ by \cite[Lemma 6.16]{KV}. It is well-known (see \cite{Weibel}) that for any group $G$ and any trivial $G$-module $A$, $H^1(G,A) = \Hom(G,A)$, so for $S \in \mathcal{T}^{F_3}_k$, $H^1(\hat{\mathbb{Z}},S) = S$ and hence $$\sum_{S \in \mathcal{S}^{F_3}_k} |H^1(\hat{\mathbb{Z}},S)|-1 \geq \sum_{S \in \mathcal{T}^{F_3}_k} |H^1(\hat{\mathbb{Z}},S)|-1 = |\mathcal{T}^{F_3}_k|^2 - |\mathcal{T}^{F_3}_k|$$ grows faster than polynomially in $k$, as required.

Finally, for $n=2$, we note by \cite[Theorem 3.6.2]{RZ} that proper open subgroups of $F_2$ are free profinite groups of higher rank, which therefore do not have type $\PFP_1$; we conclude $F_2$ does not have type $\PFP_1$ by Proposition \ref{prop:PFPncommensurability}.
\end{proof}

\subsection{Type \texorpdfstring{$\PFP_n$}{PFPn} but not type \texorpdfstring{$\PFP_{n+1}$}{PFPn+1}}

In \cite[Proposition 4.6]{CC2}, a family $\{A_n\}$ of pro-$\mathcal{C}$ groups of type $\FP_n$ but not $\FP_{n+1}$ is constructed over $\mathbb{Z}_{\hat{\mathcal{C}}}$, the pro-$\mathcal{C}$ completion of $\hat{\mathbb{Z}}$, for any class $\mathcal{C}$ of finite groups closed under subgroups, quotients and extensions.

\begin{prop}
\label{FPnnotn+1}
Let $\mathcal{C}$ be the class of finite soluble groups, so that $\mathbb{Z}_{\hat{\mathcal{C}}} = \hat{\mathbb{Z}}$. Then $A_n$ has type $\PFP_n$ but not type $\PFP_{n+1}$ over $\hat{\mathbb{Z}}$.
\end{prop}
\begin{proof}
By Remark \ref{ex:PFPnprosoluble}, for finitely generated prosoluble groups, type $\PFP_n$ over $R$ is equivalent to type $\FP_n$ over $R$. Since $A_n$ is finitely generated for $n \geq 1$, we conclude that $A_n$ has type $\PFP_n$ but not type $\PFP_{n+1}$ over $\hat{\mathbb{Z}}$ for $n \geq 1$. For $n=0$, $A_0$ is not finitely generated, hence not of type $\FP_1$ over $\hat{\mathbb{Z}}$ by \cite[Corollary 2.4]{Damian}, hence not of type $\PFP_1$.
\end{proof}

As an additional example, we consider the iterated wreath products described in \cite{Vannacci}. Let $G$ be the infinitely iterated wreath product of copies of the alternating group $A_{36}$ defined in \cite[Remark 2]{Vannacci}. $G$ is PFG by the proof of \cite[Theorem A]{Quick}, so it has type $\PFP_1$ over $\hat{\mathbb{Z}}$, and has UBERG -- though it is not finitely presented, by \cite[Remark 2]{Vannacci}.

\begin{prop}
$G$ does not have type $\PFP_2$ over $\hat{\mathbb{Z}}$.
\end{prop}
\begin{proof}
Since $G$ has UBERG, it is enough to show $G$ does not have type $\FP_2$. $A_{36}$ has Schur multiplier of size $2$, so for $W_n$ the iterated wreath product of copies of $A_{36}$, iterated $n$ times, we get $H_2(W_n,\hat{\mathbb{Z}}) = (\mathbb{Z}/2\mathbb{Z})^n$ by \cite[Theorem 3]{Read}. Taking inverse limits over $n$, $H_2(G,\hat{\mathbb{Z}})$ is an infinite product of copies of $\mathbb{Z}/2\mathbb{Z}$. In particular, it is not finitely generated, so $G$ does not have type $\FP_2$ by \cite[Lemma 4.5]{CC2}.
\end{proof}


\begin{thebibliography}{99}
\bibitem{AG}
M.~Aschbacher and  R.~Guralnick. 
\newblock Some applications of the first cohomology group.
\newblock {\em J. Algebra}  90(2):446--460, 1984.

\bibitem{BT}
F.~Beyle and J.~Tappe.
\newblock Group extensions, representations, and the Schur multiplicator, volume 958 of {\em Lecture Notes in Mathematics}.
\newblock Springer-Verlag, 1982.

\bibitem{Bieri}
R.~Bieri.
\newblock Homological dimension of discrete groups, in \emph{Queen Mary College Mathematical Notes}.
\newblock Queen Mary College, London 1981.

\bibitem{Brown}
K.S.~Brown.
\newblock Cohomology of groups, volume 87 of {\em Graduate Texts in Mathematics}.
\newblock Springer-Verlag, 1994.

\bibitem{CC2}
G.~Corob Cook.
\newblock Bieri-Eckmann criteria for profinite groups.
\newblock {\em Israel J.\ Math.}, 212:857--893, 2016.

\bibitem{CC}
G.~Corob Cook.
\newblock On profinite groups of type $\FP_\infty$.
\newblock {\em Adv.\ Math.}, 294:216--255, 2016.

\bibitem{CGK}
J.~Cossey, K.W.~Gruenberg, L.G.~Kov\'{a}cs
\newblock The presentation rank of a direct product of finite groups.
\newblock {\em Journal of Algebra}, 28(3):597--603, 1974.


\bibitem{Damian}
E.~Damian.
\newblock The generation of the augmentation ideal in profinite groups.
\newblock {\em Israel J.\ Math.}, 186:447--476, 2011.

\bibitem{FJ}
M.D.~Fried and M.~Jarden.
\newblock Field arithmetic, volume 11 of {\em A Series of Modern Surveys in Mathematics}.
\newblock Springer-Verlag, Berlin, 2008.

\bibitem{Gruenberg}
K.W.~Gruenberg.
\newblock Relation Modules of Finite Groups, Conference Board of the Mathematical Sciences Regional Conference Series in Mathematics, No.\ 25.
\newblock American Mathematical Society, Providence, RI, 1976.

\bibitem{GKKL}
R.~Guralnick, W.~Kantor, M.~Kassabov and A.~Lubotzky.
\newblock Presentations of finite simple groups: profinite and cohomological approaches.
\newblock {\em Groups Geom.\ Dyn.}, 1:469--523, 2007.



\bibitem{Hill}
V.~Hill.
\newblock Split and minimal abelian extensions of finite groups.
\newblock {\em Trans.\ Amer.\ Math.\ Soc.}, 172:329--337, 1972.

\bibitem{KaLu}
W.~Kantor and A.~Lubotzky.
\newblock The probability of generating a finite classical group.
\newblock {\em Geom.\ Dedicata}, 36(1):67--87, 1990.

\bibitem{KL}
P.~Kleidman and M.~Liebeck. 
\newblock {\em The subgroup structure of the finite classical groups}, LMS Lecture Note Series 129.
\newblock Cambridge University Press, Cambridge, 1990.



\bibitem{KarpSM}
G.~Karpilovsky.
\newblock {\em The Schur Multiplier}, LMS Monographs 2.
\newblock Oxford University Press, Oxford, 1987.


\bibitem{KV}
S.~Kionke and M.~Vannacci.
\newblock Positively finitely related profinite groups. 
\newblock {\em Israel J.\ Math.}, 225(2):743--770, 2018.




\bibitem{Lubotzky}
A.~Lubotzky.
\newblock Pro-finite presentations.
\newblock {\em J.\ Algebra}, 242:672--690, 2001.

\bibitem{LS}
A.~Lubotzky and D.~Segal.
\newblock {\em Subgroup growth}, volume 212 of {\em Progr.\ Math}.
\newblock Birkh\"auser Verlag, Basel, 2003.

\bibitem{Mann}
A.~Mann.
\newblock Positively finitely generated groups.
\newblock {\em Forum Math.} 8(8):429--460, 1996.

\bibitem{MS}
A.~Mann and A.~Shalev.
\newblock Simple groups, maximal subgroups, and probabilistic aspects of profinite groups.
\newblock {\em Israel J.\ Math.}, 96:449--468, 1996.

\bibitem{MRD}
L.~Morgan and C.~Roney-Dougal.
\newblock A note on the probability of generating alternating or symmetric groups.
\newblock {\em Arch.\ Math.}, 105:201--204, 2015.

\bibitem{Quick}
M.~Quick.
\newblock Probabilistic generation of wreath products of non-abelian finite simple groups, II.
\newblock {\em Internat.\ J.\ Algebra Comput.}, 16(3):493--503, 2006.

\bibitem{Read}
E.W.~Read.
\newblock On the Schur multiplier of a wreath product.
\newblock {\em Illinois J.\ Math.}, 20(3):456--466, 1976.

\bibitem{RZ}
L.~Ribes and P.~Zalesskii.
\newblock Profinite groups, volume 40 of \emph{A Series of Modern Surveys in Mathematics}.
\newblock Springer-Verlag, Berlin, 2010.

\bibitem{SW}
P.~Symonds and T.~Weigel.
\newblock Cohomology of $p$-adic analytic groups. In {\em New horizons in pro-{$p$} groups}, volume 184 of {\em Progr.\ Math.}, pages 347--408.
\newblock Birkh\"auser Boston, Boston, MA, 2000.

\bibitem{Vannacci}
M.~Vannacci.
\newblock On hereditarily just infinite profinite groups
obtained via iterated wreath products.
\newblock {\em J.\ Group Theory} 19:233--238, 2016.

\bibitem{Weibel}
C~Weibel. 
\newblock An introduction to homological algebra, volume 38  of \emph{Cambridge Studies in Advanced Mathematics}. 
\newblock Cambridge University Press, Cambridge, 1994.

\bibitem{Wilson}
J.S.~Wilson.
\newblock Profinite groups, volume 19 of \emph{London Mathematical Society Monographs, New Series}.
\newblock Oxford University Press, New York, 1998.

\end{thebibliography}
\end{document}